\documentclass[a4paper, 11pt,reqno]{amsart}
\usepackage{amsthm,amssymb,amsmath,graphicx,psfrag,enumitem,mathrsfs}
\usepackage{mathtools}
\usepackage{dsfont}
\usepackage{bm}
\usepackage[foot]{amsaddr}
\usepackage[british]{babel}
\usepackage[babel]{microtype}
\usepackage[margin=1in]{geometry}
\usepackage{setspace}

\usepackage[textsize=footnotesize]{todonotes}

\usepackage{algorithm}
\usepackage{algpseudocode}
\usepackage{tikz}
\usetikzlibrary{calc} 

\usetikzlibrary{backgrounds}

\usepackage[initials]{amsrefs} 

\makeatletter
\def\NAT@spacechar{~}% NEW
\makeatother

\usepackage{hyperref}
\usepackage[capitalise]{cleveref}
\hypersetup{colorlinks=true,
   citecolor=blue,
   filecolor=blue,
   linkcolor=blue,
   urlcolor=blue
}

\crefname{figure}{Figure}{Figures}
\Crefname{figure}{Figure}{Figures}

\allowdisplaybreaks

\newtheorem{definition}{Definition}[section]
\newtheorem{claim}{Claim}
\newtheorem{construction}{Construction}

\newtheorem{proposition}[definition]{Proposition}
\newtheorem{theorem}[definition]{Theorem}

\newtheorem{lemma}[definition]{Lemma}

\newtheorem{question}[definition]{Question}
\newtheorem{problem}[definition]{Problem}

\numberwithin{equation}{section}

\newcommand{\comment}[1]{}

\newcommand{\cH}{\mathcal{H}}

\newcommand{\dist}{{\rm dist}}

\newcommand{\Pro}{\mathbb{P}}

\renewcommand{\epsilon}{\varepsilon}

\newcommand{\ora}{\overrightarrow}

\newcommand{\strm}{{\rm sm}}
\newcommand{\COMMENT}[1]{}
\renewcommand{\COMMENT}[1]{\footnote{\textcolor{blue!70!black}{#1}}} % comment out to hide comments

\title{Strong complete minors in digraphs}
\date{\today}

\address{ $^2$ School of Mathematics, University of Birmingham, 
Edgbaston, Birmingham, B15 2TT, United Kingdom.}
\address{$^1$ Department of Mathematics, Karlsruhe Institute of Technology,  76131 Karlsruhe, Germany.}

\author[M.~Axenovich]{Maria Axenovich $^1$}

\email{maria.aksenovich@kit.edu}
\author[A.~Gir\~{a}o]{Ant\'onio Gir\~{a}o $^2$}
\email{giraoa@bham.ac.uk}
\author[R.~Snyder]{Richard Snyder $^1$}
\email{richard.snyder@kit.edu}
\author[L.~Weber]{Lea Weber $^1$}
\email{lea.weber@kit.edu}

\thanks{
$^2$ The research leading to these results was also partially supported by the EPSRC, grant no.~EP/N019504/1 (A.~Gir\~ao).}

\begin{document}

\begin{abstract}
Kostochka and Thomason independently showed that any graph with average degree $\Omega(r\sqrt{\log r})$ contains a $K_r$ minor. In particular, any graph with chromatic number $\Omega(r\sqrt{\log r})$ contains a $K_r$ minor, a partial result towards Hadwiger's famous conjecture. In this paper, we investigate analogues of these results in the directed setting. There are several ways to define a minor in a digraph. One natural way is as follows. A \emph{strong $\ora{K}_r$ minor} is a digraph whose vertex set is partitioned into $r$ parts such that each part induces a strongly-connected subdigraph, and there is at least one edge in each direction between any two distinct parts. We investigate bounds on the dichromatic number and minimum out-degree of a digraph that force the existence of strong $\ora{K}_r$ minors as subdigraphs. In particular, we show that any tournament with dichromatic number at least $2r$ contains a strong $\ora{K}_r$ minor, and any tournament with minimum out-degree $\Omega(r\sqrt{\log r})$ also contains a strong $\ora{K}_r$ minor. The latter result is tight up to the implied constant, and may be viewed as a strong-minor analogue to the classical result of Kostochka and Thomason. Lastly, we show that there is no function $f: \mathbb{N} \rightarrow \mathbb{N}$ such that any digraph with minimum out-degree at least $f(r)$ contains a strong $\ora{K}_r$ minor, but such a function exists when considering dichromatic number.
\end{abstract}
\maketitle

%\onehalfspacing

\section{Introduction}

The relationship between the chromatic number of a graph and the existence of certain types of substructures, like minors and topological minors, has a long history. As usual, we say that a graph is a $K_r$ \emph{minor} if there is a partition of its vertex set into $r$ parts such that each part induces a connected subgraph, and there is at least one edge between any two distinct parts. One of the most famous examples of a problem of the aforementioned type is \emph{Hadwiger's conjecture}: for every $r \ge 1$, any graph with chromatic number at least $r$ contains a $K_r$ minor. 

There are a few partial results concerning this conjecture. It is known to be true for $r \le 6$ (see~\cites{D, H, RST, W}). For general $r$, Kostochka~\cite{K} and Thomason~\cite{Thomason} independently showed that any graph with average degree $\Omega(r\sqrt{\log r})$ contains a $K_r$ minor. Of course, this implies that any graph with chromatic number $\Omega(r\sqrt{\log r})$ contains a $K_r$ minor. This was the best general result towards Hadwiger's conjecture for some time. Recently, however, Norin, Postle, and Song~\cites{NS, Postle1} improved this result. Building off of that work, Postle~\cite{Postle2} proved that any graph with chromatic number $\Omega(r(\log \log r)^6)$ contains a $K_r$ minor, and this is the best bound to date.

In this paper, we look at analogous problems in the directed setting. In particular, we investigate the existence of certain types of minors in digraphs under conditions such as large dichromatic number and large minimum out-degree. Before proceeding, let us introduce a bit of terminology. We remark that all digraphs we consider are simple and do not contain loops. A digraph $D$ is \emph{strongly-connected} if for every ordered pair $(u, v)$ of vertices in $D$ there is a directed path in $D$ from $u$ to $v$. As usual, $D$ is \emph{weakly-connected} if its underlying graph is connected. The complete directed graph on $r$ vertices, denoted by $\ora{K}_r$, is a directed graph in which every pair of vertices is connected by an edge in each direction.

There are several ways one can define a minor in a digraph (e.g., the notion of a \emph{butterfly minor}, originally introduced by Johnson, Robertson, Seymour, and Thomas; see section $5$ of~\cite{JRST}). Here, we consider so-called \emph{strong minors}. Given a digraph $H$, we say that $D$ is a \emph{strong $H$ minor} if $V(D)$ admits a partition $\{X_v: v \in V(H)\}$ into non-empty sets (called \emph{branch sets}) such that 
\begin{itemize}
   \item the digraph $D[X_v]$ induced by $X_v$ is strongly-connected for all $ v \in H$, and
   \item $uv \in E(H)$ if and only if there is an edge in $D$ from $X_u$ to $X_v$.
\end{itemize}
A similar definition holds for \emph{weak $H$ minors}: we merely insist in this case that each branch set induces a weakly-connected subdigraph. We say that a digraph $D$ \emph{contains} a strong (weak) $H$ minor if it contains a strong (weak) $H$ minor as a subdigraph. Equivalently, $D$ contains a strong $H$ minor whenever $H$ can be obtained from a subdigraph of $D$ by repeatedly contracting a strongly-connected subdigraph to a vertex and removing loops and multiple edges. We remark that the notion of a strong minor has been investigated before by Kim and Seymour~\cite{KS}, where they showed that tournaments (more generally, semi-complete digraphs) are well-quasi-ordered under the strong minor relation.

Jagger (\cite{J},~\cite{J2}) investigated extensions of the classical result of Kostochka and Thomason to digraphs. In particular, he addressed the following question: How many edges must a digraph have in order to guarantee weak (or strong) complete minors as subdigraphs? In particular, he showed~\cite{J} that any digraph with average degree $\Omega(r\sqrt{\log r})$ contains a weak $\ora{K}_r$ minor. This is a weak-minor analogue of the aforementioned result of Kostochka and Thomason. A similar theorem, however, cannot hold for containing \emph{strong} $\ora{K}_r$ minors: the transitive tournament on $n$ vertices has $\binom{n}{2}$ edges, but does not even contain a strong $\ora{K}_2$ minor. Thus, in order for a digraph $D$ to contain a strong $\ora{K}_r$ minor, it must have more than $\binom{|D|}{2}$ edges. Another way of saying this is the following: There is no function $f: \mathbb{N} \rightarrow \mathbb{N}$ such that any digraph with average degree at least $f(r)$ contains a strong $\ora{K}_r$ minor. This suggests that density is not the appropriate digraph parameter to force strong complete minors, but perhaps there are other parameters that do so. 

A $k$-\emph{colouring} of a digraph is a partition of its vertex set into $k$ acyclic sets. The minimum $k$ for which this is possible is the \emph{dichromatic number} of $D$, which we shall denote by $\chi(D)$ (as we never consider the usual chromatic number in this paper, we hope this causes no confusion). This parameter was introduced by Neumann-Lara~\cite{N}, and has garnered interest in recent years (e.g., see~\cites{TandC, HLNT, HLTW} for some interesting results concerning this parameter).

While the transitive tournament discussed above is dense, it contains no large complete minors. On the other hand, there are two parameters for which the transitive tournament is essentially trivial: it only has minimum out-degree $0$ and dichromatic number $1$. Therefore, one might hope that large dichromatic number/out-degree is sufficient to guarantee large complete minors in general digraphs. And, if this is not the case, perhaps it is true in the more restrictive class of tournaments. In general, which digraph parameters, if sufficiently large, force large strong complete minors? This is the kind of question we address in this paper, our focus being on large dichromatic number and large minimum out-degree. We remark that this type of problem has been addressed in the context of forcing \emph{subdivisions} in digraphs, instead of minors; see Aboulker, Cohen, Havet, Lochet, Moura, and Thomass{\'e}~\cite{ACH}, and the recent results of Gishboliner, Steiner, and Szab{\'o} \cites{GSS, GSS2}.

\subsection{Our results}\label{subsec:results}

We shall introduce the following terminology in order to state our results:
\begin{definition}\label{defn:strong}
For a digraph $D$, let $\strm(D)$ denote the largest $r$ such that $D$ contains a strong $\ora{K}_r$ minor.
\end{definition}

Our aim is to determine whether  large dichromatic number or large out-degree is sufficient to force the existence of strong $\ora{K}_r$ minors in digraphs. 
We show that the large dichromatic number indeed guarantees the existence of strong $\ora{K}_r$ minors in digraphs.  
%Moreover, in the class of tournaments we obtain better bounds: dichromatic number that is \emph{linear} in $r$ is sufficient. 
However, we show that the large minimum out degree  is not sufficient to force strong minors in general digraphs, but it is sufficient in  tournaments.  
Specifically, we extend the classical results of Kostochka and Thomason to strong minors: any tournament with minimum out-degree $\Omega(r\sqrt{\log r})$ contains a strong $\ora{K}_r$ minor. 

%\textcolor{red}{As we have so many results, I think it would be nice for the reader to understand what are the main and probably most important resuts of this paper} 

We first consider tournaments. The following theorem asserts that dichromatic number linear in $r$ already forces strong $\ora{K}_r$ minors in tournaments.

\begin{theorem}\label{thm:main-tourn-1}
Let $r \ge 1$ be an integer and suppose $T$ is a tournament with $\chi(T) \ge 2r$. Then $\strm(T) \ge r$. Moreover, for every $r \ge 2$ there exists a tournament $S_r$ such that $\chi(S_r) = r$ and $\strm(S_r) \le r - 1$.
\end{theorem}

The construction of $S_r$ is done in \Cref{sec:const}. We believe that this construction is closer to the truth, concerning the correct dependence on $\chi$ for finding strong complete minors.

As is customary, for a digraph $D$ we denote by $\delta^+(D)$ the minimum out-degree of $D$. Our next theorem investigates strong complete minors in tournaments with large minimum out-degree. It may be viewed as the appropriate analogue of the classical result of Kostochka and Thomason for strong complete minors.

\begin{theorem}\label{thm:main-tourn-2}
There exists a constant $C > 0$ such that the following holds. If $r$ is a positive integer and $T$ is a tournament with $\delta^+(T) \ge Cr\sqrt{\log r}$, then $\strm(T) \ge r$. Moreover, this is tight up to the constant $C$.
\end{theorem}

The tightness can be seen by considering a random tournament on $cr\sqrt{\log r}$ vertices, for some constant $c > 0$, and applying a standard argument (e.g., see Bollob{\'a}s, Catlin and Erd\H{o}s~\cite{BCH}). In fact, this argument shows that the random tournament on $cr\sqrt{\log r}$ vertices with high probability does not even contain $r$ nonempty sets with an edge in each direction between each pair of sets. The analogous result to \Cref{thm:main-tourn-2} for digraphs is false, however: 

\begin{theorem}\label{thm:negative-digraph}
There is no  function $f: \mathbb{N} \rightarrow \mathbb{N}$ such that every digraph $D$ with $\delta^+(D) \ge f(r)$ satisfies $\strm(D) \ge r$.
\end{theorem}

In particular, we show that a construction due to Thomassen~\cite{Thomassen} has large minimum out-degree, but does not even contain strong $\ora{K}_3$ minors; we prove this in \Cref{sec:const} (see \Cref{prop:digraphs-outdeg}). On the other hand, we can show that large dichromatic number is sufficient to force strong minors in general digraphs. 

\begin{theorem}\label{thm:main-digraph}
Let $r \ge 1$ be an integer. If $D$ is a digraph with $\chi(D) \ge r4^r$, then $\strm(D) \ge r$. 
\end{theorem}

Note that according to the second statement of Theorem \ref{thm:main-tourn-1} there is a digraph $S_r$ such that $\chi(S_r) = r$ and $\strm(S_r) \le r - 1$. The digraph $S_r$ is a tournament.
In addition, we show in \Cref{sec:const} that for every $r \ge 2$ there exists a digraph $G_r$ with $\chi(G_r) = r$ and $\strm(G_r) \le r - 1$. In comparison with $S_r$, the digraph $G_r$ is quite sparse.
%\textcolor{red}{It is not clear to me why we would have a construction $G_r$ if the other construction for tournaments (which in particular are digraphs) already give the exact same bound. Or at least we should perhaps add a line about it?}
We believe the exponential dependence on $r$ in the first part of this theorem is far from the truth.
% The construction of $G_r$ is performed in \Cref{sec:const}.

\subsection{Organization}
The remainder of this paper is organized as follows. \Cref{sec:notation} is devoted to additional definitions and notation.  In \Cref{sec:const} we analyze constructions of a tournament and digraph that have dichromatic number $r$, but do not contain strong clique minors on $r$ vertices. We additionally prove \Cref{thm:negative-digraph} by analyzing Thomassen's construction. In \Cref{sec:tourn} we prove \Cref{thm:main-tourn-1} and \Cref{thm:main-tourn-2}, which concern finding strong minors in tournaments under the assumption of large dichromatic number and large minimum out-degree, respectively. In \Cref{sec:digraphs}, we prove our results concerning digraphs. First, we show that the assumption of large minimum out-degree is not sufficient to guarantee strong clique minors. Second, we prove \Cref{thm:main-digraph}, which shows that large dichromatic number is sufficient. Finally, we conclude in \Cref{sec:final} with some remarks and open problems.

\section{Notation and terminology}\label{sec:notation}
Here we provide some additional notation and terminology that will be used throughout the paper. Any further notation shall be introduced as necessary.

Let $D$ be a digraph. 
We denote by $V(D)$ and $E(D)$ the vertex and edge set of  $D$, respectively.
 For a vertex $v$ in  $D$, we write $N^+_D(v)$ and $N^-_D(v)$ for the out-neighbourhood and in-neighbourhood of $v$ in $D$, respectively. We let $d^+_D(v) = |N^+_D(v)|$ and $d^-_D(v) = |N^-_D(v)|$ denote the out-degree and in-degree of $v$. We let $\delta^+(D) = \min_{v \in V(D)} d^+_D(v)$ denote the minimum out-degree of $D$. We shall always omit the subscript `$D$' when the digraph is clear from context. Given a subset $X \subseteq V(D)$, we denote by $D[X]$ the subdigraph of $D$ induced by $X$.  We say that a subset of vertices $X \subseteq V(D)$ is \emph{acyclic} if the subdigraph $D[X]$ induced by $X$ contains no directed cycle. Alternatively, we may say that $X$ (or the induced subdigraph $D[X]$) is \emph{transitive}.

For disjoint sets $X, Y \subseteq V(D)$, we write $X \rightarrow_D Y$ provided every edge of $D$ with an endpoint in $X$ and an endpoint in $Y$ is oriented from $X$ to $Y$. Thus, if $D$ is a tournament, this means that every possible edge between $X$ and $Y$ is oriented from $X$ to $Y$. If $X = \{x\}$, then we simply write $x \rightarrow_D Y$ (and similarly if $Y = \{y\}$). We shall omit the subscript `$D$' when the digraph is understood from context.

A set  $X \subset V(D)$ is a \emph{cut-set} if $D - X$ is not strongly-connected. A \emph{strongly-connected component} in a digraph $D$ is a strongly-connected subdigraph of $D$ that is maximal with this property. Every digraph can be partitioned into its strongly-connected components, say $V(D) = S_1 \cup \cdots \cup S_t$ where $D[S_i]$ is strongly-connected for each $i \in [t]$ and $S_i \rightarrow S_j$ for all $1 \le i < j \le t$ (see \Cref{lem:components}). If $t \ge 2$, we then have $S := \bigcup_{i < t} S_{i} \rightarrow S_{t}$. We call $S$ the \emph{source set} and $S_{t}$ the \emph{sink set} of $D$. Moreover, if $X = \{x_1, \ldots, x_t\} \subseteq V(D)$ is a transitive set with $x_i \rightarrow x_j$ for all $1 \le i < j \le t$, we refer to $x_{1}$ and $x_{t}$ as the \emph{source} and \emph{sink} of $X$, respectively.

Finally, given a positive integer $k$, a digraph is $k$-strongly-connected if it has at least $k + 1$ vertices and if it remains strongly-connected upon the removal of any set of at most $k-1$ vertices.

%%%%%%%%%%%%%%%%%%%%%%%%%%%%%%%%%%%%%%%%%%%%%%%%%%%%%%%%%%%%%%%%%%%%%%%
%%%%%%%%%%%%%%%%%%%%%%%%%%%%%%%%%%%%%%%%%%%%%%%%%%%%%%%%%%%%%%%%%%%%%%%

\section{Preliminaries}
Here we collect some simple results that will be applied in several places throughout the paper. The first allows us to define source sets and sink sets of non-strongly-connected digraphs. The second allows us to assume that our digraphs are strongly-connected when considering dichromatic number. 

\begin{lemma}\label{lem:components}
The vertex set of every digraph $D$ can be partitioned into nonempty sets $S_1, \ldots, S_t$ such that $D[S_i]$ is strongly-connected for every $i$, and $S_i \rightarrow S_j$ for all $1 \le i < j \le t$. 
\end{lemma}
\begin{proof}
Say that two vertices $x, y$ are strongly-connected if there is a directed path from $x$ to $y$ and a directed path from $y$ to $x$ in $D$. This is clearly an equivalence relation, and we may take $S_1, \ldots, S_t$ to be the pairwise-disjoint equivalence classes. For the second claim, we may assume $t \ge 2$, since if $t=1$ it holds vacuously. Define a digraph $H$ whose vertex set is $[t]$ and we join $i$ to $j$ whenever there is an edge from $S_i$ to $S_j$ in $D$. If we cannot order the $S_i$'s transitively as claimed, then $H$ contains a directed cycle $C = i_1i_2\ldots i_ki_1$ for some $k \ge 2$. But then $D[\bigcup_{j \in V(C)} S_j]$ is strongly-connected, which contradicts the maximality of $S_1, \ldots, S_t$. 
\end{proof}

\begin{lemma}\label{lem:chi-subgraph}
If $D$ is a digraph with $\chi(D) \ge r$, then $D$ contains a strongly-connected subdigraph $D'$ with $\chi(D') \ge r$.
\end{lemma}
\begin{proof}
    If $D$ itself is strongly-connected, then we are done. Otherwise, apply \Cref{lem:components} and let $V(D) = S_1 \cup \cdots \cup S_t$ be a partition of the vertex set of $D$ into strongly-connected components, such that $S_i \to S_j$ for all $1 \le i < j \le t$, and $t \ge 2$. Assume that we have $\chi(D[S_i]) \le r-1$ for all $i \in [t]$. Colour the vertices in each $S_i$ with colours in $[r-1]$. Then this produces an $(r-1)$-colouring of $D$, a contradiction. 
\end{proof}

%%%%%%%%%%%%%%%%%%%%%%%%%%%%%%%%%%%%%%%%%%%%%%%%%%%%%%%%%%%%%%
%%%%%%%%%%%%%%%%%%%%%%%%%%%%%%%%%%%%%%%%%%%%%%%%%%%%%%%%%%%%%%
\section{constructions}\label{sec:const}

In this section, we examine the following constructions: a tournament $S_r$ with $\chi(S_r) = r$ and $\strm(S_r) \le r - 1$, a digraph $G_r$ with $\chi(G_r)= r$ and $\strm(G_r) \le r-1$, and a digraph $D_k$ with out-degree $k$ and no strong $\ora{K}_3$ minor.

\subsection{Digraphs with large dichromatic number and no large strong minor}\label{subsec:const-tourn}
Given digraphs $H, G, D$, let $\Delta(H, G, D)$ be the digraph obtained by taking vertex-disjoint copies of $H$, $G$, and $D$, 
%forming a complete tripartite graph with parts $H$, $G$, $D$, 
adding all possible undirected edges with at most one endpoint in each of $V(H), V(G),$ and $V(D)$, and orienting the new edges so that $H \rightarrow G$, $G\rightarrow D$, and $D \rightarrow H$.
\begin{construction}\label{const:tourn}
Let $S_1 = K_1$ (i.e., the single vertex complete digraph with no edges). For $r \ge 2$ we set $S_r = \Delta(K_1, S_{r-1}, S_{r-1})$.
\end{construction} 
We remark that this construction has appeared before in~\cite{TandC}. It is not difficult to check by induction that $\chi(S_r) = r$ for all $r \ge 1$. The following shows that, in general, dichromatic number $r$ is not sufficient to guarantee a strong $\ora{K}_r$ minor.
\begin{figure}
	\begin{tikzpicture}[scale=0.5]
	%\draw[step=1.0,gray,very thin] (-1,0) grid (10,10);

	\draw  (4.5, 6) circle (2pt);
	\node at (5,6) {};

	\draw (1.7,3.1) arc(110:15:1.1cm and 1.1cm) -- (4.4,5.85) -- (1.7,3.1) ;
	\draw (2,2) circle (1cm);
	\draw[->, very thick]  (3.85,5) -- (2.75,3.2) ;
	\node at (2,2) {$S_{r-1}$};
	
	\draw (4.6,5.85) -- (7.3,3.1)  arc(70:165:1.1cm and 1.1cm) -- (4.6,5.85) ;
	\draw (7,2) circle (1cm);
	\draw[->, very thick]  (6.25,3.2) --(5.15,5);
	\node at (7,2) {$S_{r-1}$};
	
	\draw (3.1,2.1) arc(7:-50:1.1cm and 1.1cm) -- (6.25,1.13)  arc(230:173:1.1cm and 1.1cm) -- (3.1,2.1); 
	
	\draw[->, very thick]  (3.4,1.75) -- (5.3,1.75) ;
	\draw[->, very thick]  (3.4,1.45) -- (5.3,1.45) ;

	\end{tikzpicture}\caption{The tournament $S_r$}\end{figure}
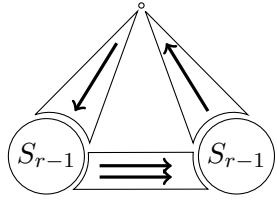
\begin{proposition}\label{prop:constr-1}
For every $r \ge 2$ we have $\chi(S_r) = r$ and ${\rm sm}(S_r) \le r - 1$ for every $r \ge 2$.
\end{proposition}
\begin{proof}
As mentioned, $\chi(S_r) = r$ is easy to verify by induction on $r$. We turn to determining $\strm(S_r)$. Note that $S_2$ is simply a directed triangle and so clearly $\strm(S_2) = 1$. Let $r \ge 3$ and suppose the result holds for smaller values. Since $S_r = \Delta(K_1, S_{r-1}, S_{r-1})$ let us write $A$ and $B$ for the two copies of $S_{r-1}$, and $v$ for the copy of $K_1$. Clearly, at most one branch set from any strong minor in $S_r$ contains $v$. If $X \subseteq V(S_r)$ is a strongly-connected subset that does not contain $v$, then $X$ cannot contain vertices from both $A$ and $B$ (since $A \rightarrow B$). Similarly, there cannot be one branch set contained in $A$ and another contained in $B$. Since we can potentially add a new branch set in $S_r$ containing $v$ we have
\[
    {\rm sm}(S_r) \le {\rm sm}(S_{r-1}) + 1 \le r - 1,
\]
as claimed.\end{proof}

%\subsection{Digraphs}\label{subsec:constr-digraphs}
\begin{construction}\label{constr:digraphs}
Let $G_1 = K_1$ and suppose $G_r$ has been constructed for some $r \ge 1$. To form $G_{r+1}$, consider first a transitive tournament $T_{r+1}$ on $r+1$ vertices. For each directed edge $e$ of $T_{r+1}$ add pairwise disjoint copies of $G_r$, denoted by $G^e_r$, such that $V(G_r^e) \cap V(T_{r+1}) = \varnothing$ for all $e$. Finally, for each directed edge $e = vw$ in $T_{r+1}$ and each vertex $u \in V(G_r^e)$ create a directed triangle $uvw$.
\end{construction}

\begin{proposition}\label{prop:constr-2}
For every $r \ge 2$ we have $\chi(G_r) = r$ and ${\rm sm}(G_r) \le r - 1$.
\end{proposition}
\begin{proof}
Note that $G_2$ is a directed triangle, so we have $\chi(G_2) = 2$ and ${\rm sm}(G_2) = 1$. So suppose $r \ge 3$ and the result holds for smaller values. Since each copy of $G_{r-1}$ in $G_r$ can be $(r-1)$-coloured, we can $r$-colour $G_r$ by using one extra colour for the transitive tournament. On the other hand, we cannot $(r-1)$-colour $G_r$: assume there is an $(r-1)$-colouring of $G_r$, then some two vertices $x, y  \in V(T_r)$  receive the same colour, say $1$. Let $e = xy$ be the corresponding edge in $T_r$. As $\chi(G_{r-1}^e) = r - 1$, the colour $1$ must appear in $G_{r-1}^e$, and so we obtain a directed triangle in colour $1$, a contradiction.

%To see that  ${\rm sm}(G_r) \geq  r - 1$, we shall show first that $G_r$ contains a strong $\ora{K}_{r-1}$ minor  that does not span all vertices of $G_r$. We use  induction on $r$ with a trivial basis for $r=1$ and $r=2$. 
%Assume that the statement holds for $G_{r-1}$ and consider $G_r$, $r\geq 3$. Consider further $G^e_{r-1}$  for $e=xy$ and  a strong $\ora{K}_{r-2}$
%minor  $K$  on vertex set $Q$, $Q\subsetneq V(G^e_{r-1})$, that exists by induction. Let $w\in V(G^e_{r-1})\setminus Q$. 
%Then the branch sets of $K$ together with $\{x,y, w\}$ form branch sets of a strong  $\ora{K}_{r}-1$ minor, that is in particular a non spanning 
%subgraph of $G_r$. 

Next, we shall show that ${\rm sm}(G_r) \leq  r - 1$. Let $t={\rm sm}(G_r)$.
Consider  a strong $\ora{K}_t$ minor in $G_r$ with branch sets $B_1, \ldots, B_t$. 
Since $T_r$ is not strongly-connected, $B_i \not\subseteq V(T_r)$ for each $i=1, \ldots, t$. 
Observe also that if  for  some edge $e=xy$ of $T_r$, $B_i\cap V(G_{r-1}^e)\neq \emptyset$ and $B_i\setminus   V(G_{r-1}^e)\neq \emptyset$, then $B_i$ contains both $x$ and $y$; otherwise, $B_i$ does not induce a strongly-connected digraph. There are two cases to consider.

\textbf{Case $1$}.  \emph{There is $i \in [t]$ such that $B_i \subseteq V(G_{r-1}^e)$ for some edge $e=xy$ of $T_r$.}

\noindent Without loss of generality, assume $i =t$. Let $V'=V(G_{r-1}^e)$. 
Then any other $B_j$, $j=1, \ldots, t-1$ should either be contained in  $V'$ or contain $x$ or $y$. 
 If all $B_i$'s are contained in $V'$, then by induction $t\leq r-2$.  
 Thus, we can assume that there is $j\in \{1, \ldots, t-1\}$ such that $B_j$ contains $x$ or $y$. If $B_j$ intersects $V'$, we have that $B_j$ contains both $x$ and $y$ by an observation before the statement of the case. If $B_j$ does not contain vertices from $V'$ it also must contain both $x$ and $y$, otherwise the edges between $B_t$ and $B_j$ go in one direction only.
 Since the $B_i$'s are pairwise disjoint,  there is exactly one such index $j$,  such that $B_j$ is not contained in $V'$. 
 The number of branch sets from $B_1, \ldots, B_t$ that are contained in $V'$ is by  induction  at most $r-2$. 
 Thus ${\rm sm}(G_r)=t \leq r-2+1 = r-1$.

\textbf{Case $2$}.  \emph{For each branch set $B_i$ there exists an edge $e$ of $T_r$ such that $B_i$ contains both vertices of $e$.}

\noindent Since the $B_i$'s are pairwise disjoint, the respective edges $e$ must also be pairwise disjoint. Since $T_r$ has at most $r/2\leq r-1$ pairwise disjoint edges, $t\leq r-1$. This completes the second case, and thus the proof.
\end{proof}

\subsection{Digraphs with large out-degree and no large strong minor}\label{subsec:digraphs-degree}

Here we show that large out-degree is not sufficient to guarantee strong $\ora{K}_r$ minors in general digraphs, and thus prove \Cref{thm:negative-digraph}. To this end, we need Thomassen's~\cite{Thomassen} construction of digraphs with large out-degree and no even cycle. The construction proceeds as follows. 

\begin{construction}\label{constr:thomassen}
Let $D_1$ be a directed triangle. Suppose $D_k$ is defined by induction. Now, add pairwise disjoint sets $A_x \cup \{x'\}$ where $|A_x| = k + 1$, for every vertex $x \in V(D_k)$, all of them disjoint from $V(D_k)$. We then add all directed edges between $A_x$ and $N_{D_k}^+(x) \cup \{x\}$ such that $A_x \rightarrow \{x\} \cup N_{D_k}^+(x)$, and $x \rightarrow x'$. Finally, add all edges from $x'$ to $A_x$, and denote by $D_{k+1}$ the resulting digraph.
\end{construction}
Then $D_{k}$ clearly has minimum out-degree $k$, and one can check that it contains no even directed cycle (in particular, it contains no pair of vertices $x, y$ with $xy, yx \in E(D_{k+1})$). The following asserts that $D_k$ does not contain large strong clique minors. In particular, large out-degree alone is not sufficient to guarantee strong minors in digraphs.
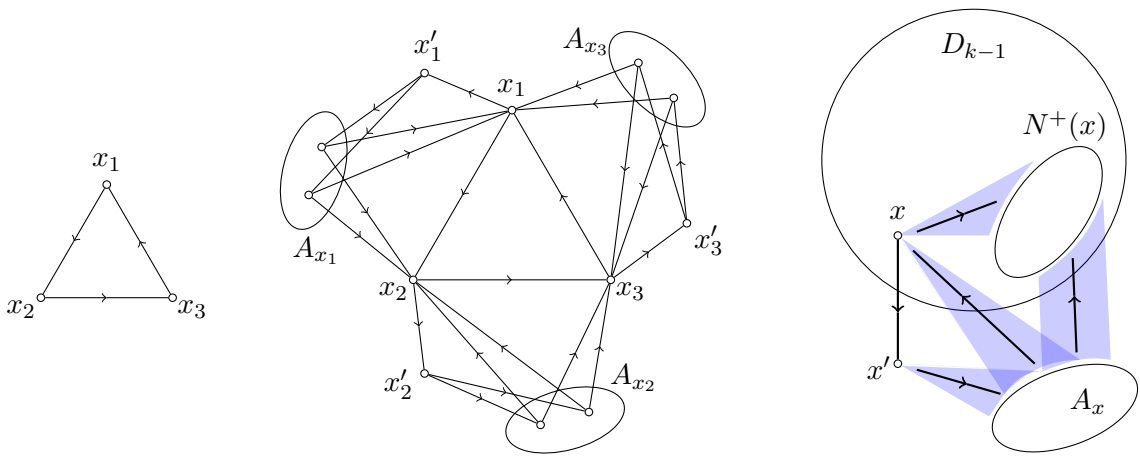
\begin{figure}[H]
	\begin{tikzpicture}
	\foreach \s in {1,...,3}
	{
		%x
		\node[draw, circle, inner sep=1pt] (x\s)at ({360/3* (\s - 1) + 90}:1cm) {};
		\node at ({360/3 * (\s - 1) + 90}:1.3cm) {$x_\s$};\
	}
	\draw[->]  ($(x1)!0.5!(x2)$)--(x2) (x1)--($(x1)!0.5!(x2)$);
	\draw[->] ($(x2)!0.5!(x3)$)--(x3) (x2)--($(x2)!0.5!(x3)$);
	\draw[->] ($(x3)!0.5!(x1)$)--(x1) (x3)--($(x3)!0.5!(x1)$);
	\node at (0,-2.5) {};
	\end{tikzpicture}\qquad
	\begin{tikzpicture}
	%\draw[step=1.0,gray,very thin] (-4,-4) grid (4,4);
	\def \n {3}
	
	\foreach \s in {1,...,\n}
	{
		%x
		\node[draw, circle, inner sep=1pt] (x\s)at ({360/\n * (\s - 1) + 90}:1.5cm) {};
		\node at ({360/\n * (\s - 1) + 90}:1.8cm) {$x_\s$};
		%x'
		\node[draw, circle, inner sep=1pt] (x'\s)at ({360/\n * (\s - 1) + 120}:2.3cm) {};
		\node at ({360/\n * (\s - 1) + 115}:2.6cm) {$x'_\s$};
		%A_x points
		\node[draw,circle, inner sep=1pt] (A\s)at ({360/\n * (\s - 1) + 158}:2.7cm)  {};
		\node[draw,circle, inner sep=1pt] (B\s)at ({360/\n * (\s - 1) + 172}:2.7cm)  {};
		% A_x ellipse
		\draw[rotate={(360/\n* (\s-1)+165)}] (0:2.7cm) ellipse (0.4cm and 0.8cm); 
		\node at ({360/\n * (\s - 1) + 188}:2.6cm) {$A_{x_\s}$};
		
		\draw[->] ($(x\s)!0.5!(x'\s) $)-- (x'\s)  (x\s)--($(x\s)!0.5!(x'\s) $);
		\draw[->] ($(x'\s)!0.5!(A\s) $)-- (A\s)  (x'\s)--($(x'\s)!0.5!(A\s) $); 
		\draw[->] ($(x'\s)!0.5!(B\s) $)-- (B\s)  (x'\s)--($(x'\s)!0.5!(B\s) $); 
		\draw[->] ($(A\s)!0.5!(x\s) $)-- (x\s)  (A\s)--($(A\s)!0.5!(x\s) $);
		\draw[->] ($(B\s)!0.5!(x\s) $)-- (x\s)  (B\s)--($(B\s)!0.5!(x\s) $);
	}
	
	\foreach \s in {x,A,B} {
		\draw[->]  ($(\s1)!0.5!(x2)$)--(x2) (\s1)--($(\s1)!0.5!(x2)$);
		\draw[->] ($(\s2)!0.5!(x3)$)--(x3) (\s2)--($(\s2)!0.5!(x3)$);
		\draw[->] ($(\s3)!0.5!(x1)$)--(x1) (\s3)--($(\s3)!0.5!(x1)$);
	}
	\end{tikzpicture}\qquad\quad
	\begin{tikzpicture}
	\def \n {3}
	
	\draw (0,0) circle (2cm);
	\node at (0,1.5) {$D_{k-1}$};
	
	\draw[rotate=-35] (1.2,0) ellipse (0.5cm and 1cm);
	\node at (1.2,0.45) {$N^+(x)$};
	
	\node[draw, circle, inner sep=1pt] (x) at (-1,-1) {};
	\node at (-1,-0.7) {$x$};
	
	\draw[rotate=-70] (3.5,0) ellipse (0.5cm and 1cm);
	\node at (1.5,-3.2) {$A_x$};
	\coordinate (ax1) at (0.4,-3.1){};
	\coordinate (ax2) at (0.9,-2.8){};
	\coordinate (ax3) at (1.4,-2.65){};
	\coordinate (ax4) at (0.2, -3.4){};
	\coordinate (ax5) at (1.8, -2.65) {};
	
	\coordinate (b1) at (1.7,-0.5) {};
	\coordinate (b2) at (0.85,-1.55) {};

	\node[draw, circle, inner sep=1pt] (x') at (-1,-2.7) {};
	\node at (-1.25,-2.75) {$x'$};
	
	\fill[blue, opacity=0.2] (ax2) to[out=180, in=55 ] (ax4) -- (x')--(ax2);

	\fill[blue, opacity=0.2](ax1) to[in=180, out=45] (ax3) -- (x) -- (ax1);
	
	\fill[blue, opacity=0.2] (b1) to[out=250, in = 35] (b2) -- (ax2) to[out=20, in = 170] (ax5) -- (b1);
	
	\fill[blue, opacity=0.2] (x) to (0.2,-1) to[out=70, in=230] (0.8,0) -- (x);
	
	\draw[->, thick]  ($(x)!0.6!(x')$)--(x') (x)--($(x)!0.6!(x')$);
	
	\draw[->, thick] ($(0.8,-2.7)!0.6!(-0.8,-1.2)$)--(-0.8,-1.2) (0.8,-2.7)--($(0.8,-2.7)!0.6!(-0.8,-1.2)$) ;
	
	\draw[->, thick] ($(-0.75,-2.78)!0.6!(0.35,-3.1)$)--(0.35,-3.1) (-0.75,-2.78)--($(-0.75,-2.78)!0.6!(0.35,-3.1)$) ;
	\draw[->, thick]   ($(-0.75, -0.95)!0.6!(0.3,-0.55)$)--(0.3,-0.55) (-0.75, -0.95)--($(-0.75, -0.95)!0.6!(0.3,-0.55)$) ;
	\draw[->, thick]  ($(1.35, -2.55)!0.6!(1.3, -1.3)$)--(1.3, -1.3) (1.35, -2.55)--($(1.35, -2.55)!0.6!(1.3, -1.3)$) ;
	
	% 	\foreach \s in {1,2,3,4}{
	% 	\draw[->, thick]($(ax\s)!0.5!(x)$)--(x) (ax\s)--($(ax\s)!0.5!(x)$);
	% 	};
	\end{tikzpicture}
	\caption{The digraph $D_k$ for $k=1, 2$ and sketch for general $k$}
\end{figure}

\begin{proposition}\label{prop:digraphs-outdeg}
For every $k \ge 1$, $D_k$ does not contain a strong $\ora{K}_3$ minor.
\end{proposition}

\begin{proof}
	
%There is a strong $\ora{K}_2$-minor in $D_k$ (if we say a single vertex is strongly-connected) for $k \ge 2$: Take branch sets $U = \{x, x', y\}$, $V = \{y'\}$ for any vertex $x \in D_{k-1}$ and $y, y' \in A_x$. 

%However, $D_k$ contains no $\ora K_3$-minor for any $k$:

We proceed by induction on $k$. This is clearly true for $k = 1$, so suppose $k \ge 2$ and the result holds for smaller values of $k$. Suppose, by way of contradiction, that $D_{k+1}$ contains a strong $\ora{K}_3$ minor with branch sets $U, W, X$. Since $D_k$ contains no strong $\ora{K}_3$ minor, it must be the case that at least one of these branch sets intersects $A_x \cup \{x'\}$ for some $x \in V(D_k)$. So suppose $U$ intersects $A_x \cup \{x'\}$. We claim that $x \in U$.

Indeed, first suppose that $x' \in U$. Since $N^-(x') = \{x\}$,  any nontrivial directed path to $x'$ passes through $x$. As $D[U]$ is strongly-connected, we either have $x \in U$ or $U = \{x'\}$. The latter cannot hold, since otherwise we must have both $x\in X$ and $x \in W$, a contradiction. Thus, we have $x \in U$.

Now suppose that $y \in U$ for some $y \in A_x$. Then since $N^-(y) = \{x'\}$, any nontrivial directed path  to $y$ includes $x'$, so either $x' \in U$ or $U = \{y\}$. In the latter case we must have both $x' \in X $ and $x' \in W$, a contradiction. In the former case we have $x \in U$, by the above argument.

In any case, $x\in U$ as claimed. But then $U' = U\cap V(D_k) $ induces a strongly-connected subdigraph in $D_k$. Indeed, if $U' = \{x\}$, then it is trivially strongly-connected in $D_k$. Otherwise, let $u, v \in U'$ be distinct vertices. Then any path from $u$ to $v$ that uses vertices not in $D_k$,  say from $A_{x} \cup \{x'\}$,  passes through $x$ and eventually an out-neighbour $z$  of $x$. So there is a path from $u$ to $v$ contained in $U'$, obtained by  going from $x$ directly to $z$ and continuing outside of $A_x \cup \{x'\}$.
Similarly, we have that $W' = W\cap V(D_k) $  and $X' = X\cap V(D_k) $ induce strongly-connected subdigraphs  in $D_k$.

Since $U, W,$ and $X$ are pairwise disjoint, then for any $x$ at most one of these three sets intersects $A_x \cup \{x'\}$. 
Since there are no edges between $A_x \cup \{x'\}$ and $A_y \cup \{y'\}$ for distinct $x$ and $y$, all edges between $U, W$, and $X$ are
the edges between $U', W'$, and $X'$. Thus $U', W',$ and $X'$ form branch sets of 
a strong $\ora{K}_3$ minor in $D_k$, a contradiction.
\end{proof}

\section{strong minors in tournaments}\label{sec:tourn}

\subsection{Large dichromatic number in tournaments}\label{subsec:tourn-chi}

The aim of this section is to prove \Cref{thm:main-tourn-1}. Before doing so, however, we shall prove the following warm-up result, which shows that dichromatic number linear in $r$ is sufficient to guarantee strong $\ora{K}_r$ minors in tournaments.

\begin{proposition}\label{prop:tournaments-chromatic-easy}
Let $r \ge 1$ be an integer and suppose $T$ is a tournament with $\chi(T)\geq 3r$. Then $\strm(T) \ge r$. 
\end{proposition}

\begin{proof}
We proceed by induction on $r$. Clearly, this is true for $r = 1$, so we may assume $r  \ge 2$ and the result holds for smaller values of $r$. Suppose $T$ is a tournament with $\chi(T) \ge 3r$. We may assume $T$ is strongly-connected; otherwise, by \Cref{lem:chi-subgraph}, we can pass to a strongly-connected component with dichromatic number at least $3r$.
Let $A \subset V(T)$ be a largest transitive subset with source $s$ and sink $t$. Since $T$ is strongly-connected there exists a directed path in $T$ from $t$ to $s$. Pick a shortest directed path $P$ from $t$ to $s$. Note that $\chi(T[V(P)]) \le 2$: by the minimality of $P$, all edges between vertices of $P$ that have distance at least $2$ on the path are oriented backwards. It follows that vertices at odd/even distances from the initial vertex produce a $2$-colouring of $P$. Then $\chi(T[A \cup V(P)]) \le 3$. Let $T' = T - (A \cup V(P))$. It follows that $\chi(T') \ge 3r - 3 = 3(r-1)$. By induction, we have that $T'$ contains a strong $\ora{K}_{r-1}$ minor.

Observe that the maximality of $A$ implies that for every vertex $v \in V(T) \setminus A$, $v$ has at least one out-neighbour and at least one in-neighbour in $A$. In particular, for each branch set of our strong $\ora{K}_{r-1}$ minor in $T'$, there are edges in both directions between $A$ and this branch set. Finally, note that $T[A \cup V(P)]$ is strongly-connected, so $A \cup V(P)$ can be taken as the $r$th branch set in a strong $\ora{K}_r$ minor in $T$, completing the induction and the proof.
\end{proof}

The proof of \Cref{prop:tournaments-chromatic-easy} follows by first considering a maximal transitive set, ensuring that this set is in and out `dominating', and then making this set strongly-connected by attaching a shortest path. Using this naive approach, we can only guarantee the the resulting set has dichromatic number at most $3$. We would like to reduce `$3$' to `$2$', and thus in order to prove \Cref{thm:main-tourn-1}, we must work a little harder. The following lemma, which may be viewed as the heart of the proof of \Cref{thm:main-tourn-1}, shows that one can find strongly-connected dominating sets of dichromatic number $2$ in any strongly-connected tournament. This is clearly tight on the dichromatic number and it is a result of independent interest. 

Before embarking on the proof, let us introduce some useful terminology. Let $D$ be a digraph, let $S \subseteq V(D)$, and let $X \subseteq V(D) \setminus S$. We say that $S$ \emph{out-dominates} (\emph{in-dominates}, resp.) $X$ if every $x \in X$ has at least one in-neighbour (out-neighbour, resp.) in $S$. If $X = V(D) \setminus S$, we say that $S$ is \emph{out(in)-dominating}. If $S$ is both out-dominating and in-dominating, we shall say that $S$ is \emph{dominating} (or, that it is a \emph{dominating set}).

\begin{lemma}\label{lem:tournaments-dominating}
Let $T$ be a strongly-connected tournament. Then there exists a subset $R\subseteq V(T)$ satisfying the following:
\begin{itemize}
    \item $T[R]$ is strongly-connected,
    \item $R$ is dominating, and
    \item $\chi(T[R]) = 2$.
\end{itemize}
\end{lemma}
\begin{proof}
We illustrate the proof in  Figure \ref{fig-dom}.
Let $x$ be a vertex in $T$ with maximum out-degree in $T$. Let $S(x)$ be a largest possible transitive subtournament with sink $x$. Denote by $y$ the source of $S(x)$. Observe that since $S(x)$ is maximum-sized, it must be out-dominating: for all $v \in V(T) \setminus S(x)$ there is $u \in S(x)$ such that $u \rightarrow v$. Moreover, as $x$ has maximum out-degree in $T$, there exists a vertex $w^* \in V(T) \setminus S(x)$ such that $x \rightarrow w^*$ and $w^* \rightarrow y$, otherwise $y$ would have a larger out-degree than $x$.

Now, let $F=\bigcap_{s \in S(x)} N^{+}(s)$. The set $F$ consists of the vertices that are \emph{not} in-dominated by $S(x)$. If $F = \varnothing$, then we are done, since then we may take $R = S(x) \cup \{w^*\}$. So we may assume that $F \neq \varnothing$. Let $F^+ \subseteq F$ be the set of vertices in $F$ that have an out-neighbour in $V(T) \setminus F$. Since $T$ is strongly-connected, it follows that $F^+ \neq \varnothing$. We need to find a set (hopefully of small dichromatic number) that in-dominates the vertices in $F$. First, we prove the following claim showing that we may partition $F$ in a useful way.

\begin{claim}\label{claim:partition}
For some $k \ge 1$ there is a partition $F = F_1 \cup \cdots \cup F_k$ such that $F_1 = F^+$ and  for every $1 \le i \le k-1$ 
\begin{enumerate}
    \item  for every $u \in F_{i+1}$ there is $z \in F_{i}$ such that $u \rightarrow z$, and\label{itm:partition-1}
    
    \item  for every $j$, $j>i+1$, $F_i \rightarrow F_j$.\label{itm:partition-2}
\end{enumerate}
\end{claim}
\begin{proof}
Let $F_1 = F^+$. Suppose $F_1, \ldots, F_i$ have been constructed for some $i \ge 1$. Let $Y = F \setminus (\bigcup_{j \in [i]} F_j)$. If $Y = \varnothing$, then stop with $k:=i$. Otherwise, let
\[
    F_{i+1} = (\bigcup_{u \in F_i}N^-(u))\cap Y.
\]
Note that $F_{i+1}$ is nonempty due to the strong-connectivity of $T$. Indeed, if $F_{i+1} = \varnothing$, then $F_i \rightarrow Y$ and since $Y \cap F_1 = \varnothing$, no vertex of $Y$ has an out-neighbour in $V(T) \setminus F$. Thus, there is no directed path from a vertex in $Y$ to a vertex in $F_1 = F^+$, a contradiction to the strong-connectivity assumption.

Also, note that
\[
    \bigcup_{j \in [i]}F_j \rightarrow Y\setminus F_{i+1}.
\]
By construction, every vertex in $F_{i+1}$ has at least one out-neighbour in $F_i$. This completes the proof of the claim.
\end{proof}

%and we can split $F=F_1\cup F_2\cup\ldots \cup F_k$ such that for every $x\in F_i$ there is $z\in F_{i-1}$ such that $x\rightarrow z$ and $F_i \rightarrow F_j$ for every $j<i-1$. 

%(It is important that $S_1$ contains some vertex of $F_1$, otherwise we could have just taken a largest transitive set within $F_1\cup F_2$)

Let $S_k$ be a largest transitive set contained in $T[F_k]$. If $F_{k-1}$ exists (i.e. $k \ge 2$), let $S = S_k \cup S_{k-1}$, where $S_{k-1} \subseteq F_{k-1}$ is as large as possible such that $S$ is transitive, and all vertices of $S_{k-1}$ lie \emph{after} the vertices in $S_k$ in the transitive order. If $k=1$, let $S=S_1$.   Let $s$ denote the sink of $S$. In particular, either $s \in F_k$ or $s \in F_{k-1}$. Intuitively, $S$ in-dominates $F$ (which we shall show later), and it remains to `connect' $S$ to $S(x)$ to find a dominating set that is strongly-connected.

To this end, consider a shortest directed path $P=s\ldots x_1$ contained in $F$ from $s$ to a vertex in $F_1$; such a path exists by \Cref{claim:partition}. Also, observe that since $P$ is a shortest path in $F$ to $F_1$, it has exactly one vertex in each $F_i$. Now, since $x_1 \in F_1 = F^+$, by definition of $F^+$ we may choose an out-neighbour $w \in V(T) \setminus (F \cup S(x))$ of $x_1$. It is possible that $w = w^*$, but this only makes the proof simpler. So we assume that $w \neq w^*$. 
%We shall further assume that $w\rightarrow w^*$ (the other case follows an analogous argument), and let $W = \{w, w^*\}$. 
Finally, let 
\[
    R=S(x)\cup S \cup V(P)\cup W.
\]
The following sequence of claims shows that $R$ has all of the desired properties (see Figure \ref{fig-dom}).

\begin{claim}\label{claim:R-strongly-connected}
$T[R]$ is strongly-connected.
\end{claim}
\begin{proof}
First, note that $T[S(x)\cup W]$ is strongly-connected. Indeed, $T[S(x) \cup \{w^*\}]$ is strongly-connected since $w^* \in N^-(y) \cap N^+(x)$. As $w \notin F$, it has an out-neighbour $v \in S(x)$, and the maximality of $S(x)$ implies that $w$ has an in-neighbour $v' \in S(x)$. It follows that $T[R]$ is strongly-connected, as claimed.
\end{proof}

\begin{claim}\label{claim:R-dominating}
$R$ is a dominating set in $T$.
\end{claim}
\begin{proof}
$S(x)$ dominates $V(T) \setminus (S(x) \cup F)$: the maximality of $S(x)$ implies that it is out-dominating, and it additionally in-dominates $V(T) \setminus F$ (since $F$ is exactly the set of vertices \emph{not} in-dominated by $S(x)$). So it suffices to show that $R$ in-dominates $F \setminus (S \cup V(P))$. In particular, we claim that $S$ in-dominates $F \setminus (S \cup V(P))$. Indeed, by (\ref{itm:partition-2}) of \Cref{claim:partition} we have
\[
\bigcup_{i \le k - 2} F_i \rightarrow S \cap F_k,\]
so it suffices to show that $S$ in-dominates $(F_k \cup F_{k-1)}) \setminus S$. Recall that $S = S_k \cup S_{k-1}$ where $S_k$ is a largest transitive set in $F_k$, and $S_{k-1}$ is a largest possible transitive extension of $S_k$ in $F_{k-1}$ such that all vertices of $S_{k-1}$ lie after $S_k$ in the transitive order. Thus, if there is some vertex $z \in (F_k \cup F_{k-1}) \setminus S$ such that $S \rightarrow z$, we obtain a contradiction with our choice of $S$. It follows that $S$ in-dominates $(F_k \cup F_{k-1}) \setminus S$, as claimed.
\end{proof}

\begin{claim}\label{claim:R-chromatic}
$\chi(T[R]) = 2$.
\end{claim}
\begin{proof}
We produce a $2$-colouring of $R$ that depends on the parity of $k$ 
%= |V(P)|$ 
and also  on whether $s$ (the sink of $S$) is in $F_k$ or $F_{k-1}$. We deal with the case when $k$ is even first, and sketch the proof when $k$ is odd, as the proof is nearly identical. Recall that as $P = s\ldots x_1$ is a shortest $s-F_1$ path in $F$, it has exactly one vertex in each $F_i$ (except for maybe $F_k$).

\textbf{Case $1$:} \emph{$k$ is even}.
First, suppose that $k=2$. If $s \in F_2$, then let one colour class be $S(x) \cup \{x_1\}$ and the other $S \cup W$. Note that the latter set is transitive since $S \subseteq F_2$ and so sends no out-edges to $V(T) \setminus F$. On the other hand, if $s \in F_1$, then let one colour class be $S(x) \cup (S \cap F_1)$ and the other $(S \cap F_2) \cup W$. 

Now, suppose $k>2$ and $x_k := s \in F_k$. Write $P = x_kx_{k-1}\ldots x_1$ with $x_i \in F_i$ for all $i \in [k]$. Let $P_{\text{odd}} = \{x_1, x_3, \ldots, x_{k-1}\}$ and $P_{\text{even}} = \{x_2, x_4, \ldots, x_k\}$. Then we can let one colour class be $S(x) \cup P_{\text{odd}}$ and the other $S \cup P_{\text{even}} \cup W$. Indeed, $S(x)$ is transitive and $S(x) \rightarrow P_{\text{odd}}$. Moreover, $T[P_{\text{odd}}]$ is transitive by (\ref{itm:partition-2}) of \Cref{claim:partition}. Similarly, since $x_k = s \in F_k$ and $k > 2$, we have $S \cup P_{\text{even}} \subseteq F \setminus F_1$ and so $W \rightarrow S \cup P_{\text{even}}$. As before, $T[S \cup P_{\text{even}}]$ is transitive and clearly so is $W$.

Lastly, we consider the case when $x_{k-1} := s \in F_{k-1}$, i.e., when $P=x_{k-1},\ldots, x_1$ 
%($x_k$ does not exist now, 
let  $P_{\text{even}} = \{x_2, \ldots, x_{k-2}\}$ and let $P_{\text{odd}}= V(P) \setminus P_{\text{even}}$. In this case, we let one colour class be $S(x) \cup P_{\text{odd}} \cup (S \cap F_{k-1})$ and the other be $(S \cap F_k) \cup P_{\text{even}} \cup W$. It is routine to check that this is indeed a $2$-colouring using the same arguments as above.

%Take one colour class to be $S(w,w^*)\cup x_2\cup x_4\ldots x_{k-2}\cup x_{k-1}$ and the other colour class to be $S_k\cup x_1\cup x_3\cup x_5\ldots x_k$.

\textbf{Case $2$:} \emph{$k$ is odd}. If $k=1$ let $S(x)\cup S$ be one colour class, and $W$ the other.
If $k >1$, then an analogous argument as in the previous case yields a $2$-colouring of $R$.
\end{proof}

Hence, $R$ satisfies all of the claimed properties. This completes the proof of the lemma.
\end{proof}

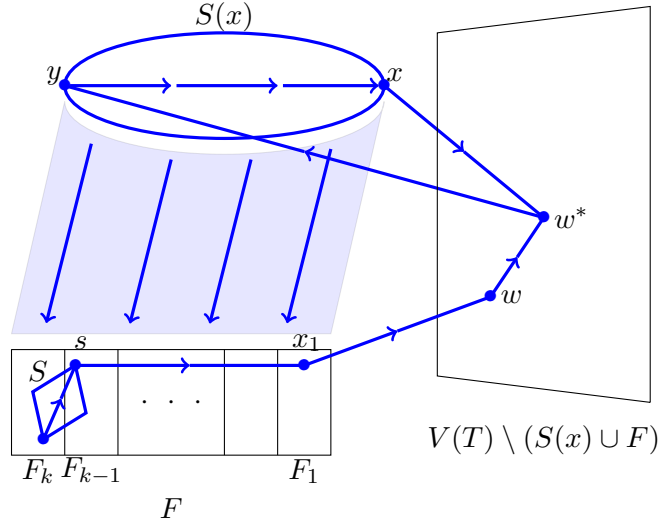
\begin{figure}[H]

\begin{tikzpicture}[scale=0.7]
%\draw[step=1.0,gray,very thin] (-1,0) grid (10,10);

%S(x) and all that belongs to it
\draw[very thick, blue] (3,9) ellipse (3cm and 1cm); 
\node at (3,10.3) {$S(x)$}; %label for S(x)
\coordinate (y) at (0,9);
\node[blue] at (y) {\textbullet}; %y
\node at (-0.2,9.2) {$y$}; %label for y
\coordinate (x) at (6,9);
\node[blue] at (x) {\textbullet}; %x
\node at (6.2,9.2) {$x$}; %label for x
\draw[->, very thick, blue] (0.1,9)--(2,9); 
\draw[->, very thick, blue] (2.1,9)--(4,9); 
\draw[->, very thick, blue] (4.1,9)--(5.9,9); 

\coordinate (w*) at (9,6.5);
\node[blue] at (w*)  {\textbullet}; %w*
\node[right] at (w*) {$w^*$}; %label for w*
\coordinate (w) at (8,5);
\node[blue] at (w)  {\textbullet}; %w
\node[right] at (w) {$w$}; %label for w
\draw[->, very thick, blue] ($(x)!0.5!(w*) $)--(w*) (x)--($ (x)!0.5!(w*) $) ; 
\draw[->, very thick, blue] ($(w*)!0.5!(y) $)--(y) (w*)--($(w*)!0.5!(y) $);
\draw[->, very thick, blue]($(w)!0.5!(w*)$)--(w*) (w)--($(w)!0.5!(w*)$); 

\draw (11,3)--(7,3.5)--(7,10)--(11,10.5)--(11,3);
\node at (9,2.3) {$V(T)\setminus (S(x) \cup F)$};

% gray arrows from S(x) to F
\draw[fill=blue, opacity=0.1] (0,8.7) arc(180:360:3cm and 1cm) -- (5,4.3) -- (-1,4.3) --(0,8.7) ; 

\draw[->, very thick, blue] (0.5,7.9) -- +(256:3.5cm);
\draw[->, very thick,blue] (2,7.6) -- +(256:3.2cm);
\draw[->, very thick,blue] (3.5,7.6) -- +(256:3.2cm);
\draw[->, very thick,blue] (5,7.8) -- +(256:3.4cm);

% F
\draw (-1,4) -- (5, 4) -- (5,2) -- (-1,2)--(-1,4);

%F_1,..,F_k
\draw (0,2)--(0,4);
\draw (1,2)--(1,4);
\node at (2.5,3)  {$\cdot$}; 
\node at (2,3)  {$\cdot$}; 
\node at (1.5,3)  {$\cdot$};
\draw (3,2)--(3,4);
\draw (4,2)--(4,4);

%S and path

\coordinate (bla) at (-0.4,2.3);
\node[blue] at (bla) {\textbullet};
\coordinate (xk) at (0.2,3.7);
\node[blue] at (xk)  {\textbullet}; %x_k
\coordinate (x1) at (4.5,3.7);
\node[blue] at (x1)  {\textbullet}; %x_1

%S
\draw[very thick, blue] (bla)--(-0.6,3.2)--(xk)--(0.4,2.8)--(bla);

\draw[->, very thick,blue]($(bla)!0.55!(xk)$)--(xk) (bla)--($(bla)!0.55!(xk)$);
%tournament up
\draw[->, very thick, blue]($(x1)!0.5!(w)$)--(w) (x1)--($(x1)!0.5!(w)$); %from x_1 to w
\draw[->, very thick, blue]($(xk)!0.5!(x1)$)--(x1) (xk)--($(xk)!0.5!(x1)$);
%from x_k to x_1

%labels
\node at (-0.5,3.6) {$S$};
\node at (0.5,1.7) {$F_{k-1}$};
\node at (-0.5,1.7) {$F_k$};
\node at (4.5,1.7) {$F_1$};
\node at (2,1) {$F$};

\node at (0.3,4.13) {$s$};
\node at (4.55,4.13) {$x_1$};

\end{tikzpicture}
\caption{The proof of \Cref{lem:tournaments-dominating}. The direction of $ww^*$ is not important for the proof.}\label{fig-dom}
\end{figure}

We can now easily deduce \Cref{thm:main-tourn-1}.

\begin{proof}[Proof of \Cref{thm:main-tourn-1}]
We have already established the second part of the statement in \Cref{prop:constr-1}. It remains to prove the first part.

We proceed by induction on $r$. Clearly, the result holds for $r = 1$, so we suppose $r \ge 2$ and the theorem holds for smaller values. Let $T$ be a tournament with $\chi(T) \ge 2r$, and, by \Cref{lem:chi-subgraph}, we may assume $T$ is strongly-connected. \Cref{lem:tournaments-dominating} grants a subset $R \subseteq V(T)$ such that $T[R]$ is strongly-connected, $R$ is dominating in $T$, and $\chi(T[R]) = 2$. Letting $T' = T - R$, we see that $\chi(T') \ge 2r - 2 = 2(r-1)$, and so by induction $T'$ contains a strong $\ora{K}_{r-1}$ minor. Taking $R$ as the $r$th branch set yields a strong $\ora{K}_r$ minor in $T$.
\end{proof}

Perhaps the correct constant in \Cref{thm:main-tourn-1} should be `$1$' instead of `$2$' as indicated by the construction $S_r$. In other words, it is possible that any tournament with dichromatic number at least $r + 1$ contains a strong $\ora{K}_r$ minor. We can prove this in the first nontrivial case:

\begin{proposition}\label{prop:small-case}
If $T$ is a tournament with $\chi(T) \ge 3$, then $\strm(T) \ge 2$.
\end{proposition}

\begin{proof}
Pick a directed triangle $C_1$ in $T$. Since $\chi(T) \ge 3$, $T - C_1$ is nonempty. If there is a vertex $v \notin C_1$ that sends edges in both directions to $C_1$, then we are done. Otherwise, $T-C_1$ can be partitioned into two sets $A_1$ and $B_1$ such that $A_1 \rightarrow C_1$ and $C_1 \rightarrow B_1$. If there is an edge $e$ from $B_1$ to $A_1$, then we are done. Indeed, one branch set is the directed triangle formed from $e$ and one vertex of $C_1$, and the other branch set is just one of the remaining vertices of $C_1$. So we may assume that $A_1 \rightarrow B_1$. In this case, we either have $\chi(A_1) \ge 3$ or $\chi(B_1) \ge 3$. Without loss of generality, $\chi(A_1) \ge 3$, and so we pass to this subtournament and repeat the argument.

In this process we obtain a sequence of nonempty sets strictly decreasing in size, so it must eventually terminate with a strong $\ora{K}_2$ minor.
\end{proof}

\subsection{Large out-degree in tournaments}\label{subsec:tourn-degree}

In this section we turn our attention to investigating the presence of strong minors in tournaments with a minimum out-degree condition. In particular, we prove \Cref{thm:main-tourn-2}. The strategy of our proof is as follows. First, we show that there is a constant $C'$ such that any tournament that is $C'r\sqrt{\log r}$-strongly-connected contains a strong $\ora{K}_r$ minor (see \Cref{lem:tournaments-connected}). Given a tournament $T$ with $\delta^+(T) \ge Cr\sqrt{\log r}$ where $C \gg C'$, we may assume that no subtournament of $T$ is $C'r\sqrt{\log r}$-strongly-connected. This assumption together with the minimum out-degree condition allows us to iteratively construct a strong $\ora{K}_r$ minor in which every branch set is, in fact, a directed triangle. This last step is made precise in \Cref{lem:triv} and \Cref{claim:main}. 

Before proceeding, we need a few preliminary results. Given a positive integer $k$ and a digraph $D$, we say that $D$ is $k$-\emph{linked} if $|V(D)| \ge 2k$ and
for any two disjoint sets of vertices $\{x_1, \ldots , x_k\}$ and $\{y_1, \ldots, y_k\}$ of $k$ vertices each,  there are pairwise vertex disjoint directed 
paths $P_1, \ldots , P_k$ such that $P_i$ has initial vertex $x_i$ and terminal vertex $y_i$ for every $i \in [k]$. We shall use the following theorem of Pokrovskiy~\cite{P}, showing that large enough strong-connectivity guarantees linkedness in tournaments.

\begin{theorem}[Pokrovskiy~\cite{P}]\label{thm:linked}
For every integer $k\ge 1$, any $452k$-strongly-connected tournament is $k$-linked.
\end{theorem}

Recall that a weak minor has the same definition as a strong minor except that we only require that the branch sets induce connected (not necessarily strongly-connected) subgraphs. As mentioned in the Introduction, Jagger~\cite{J} investigated average degree conditions for finding weak minors in digraphs. 
In the context of tournaments, we need the following:
\begin{theorem}[Jagger~\cite{J}]\label{thm:weak-minor}
There exists an absolute constant $C > 0$ such that the following holds. If $r$ is a positive integer and $T$ is a tournament with at least $Cr\sqrt{\log r}$ vertices, then $T$ contains a weak $\ora{K}_r$ minor.
\end{theorem}

During the course of the proof of \Cref{thm:main-tourn-2}, we shall construct an auxiliary (undirected) graph and apply the following classical result:

\begin{theorem}[Kostochka~\cite{K}, Thomason~\cite{Thomason}]\label{thm:kostochka-thomason}
There exists a constant $c > 0$ such that for every positive integer $r$, every graph of average degree at least $cr\sqrt{\log r}$ contains a $K_r$ minor.
\end{theorem}

Finally, we need one more preliminary result. We say that a subset $B \subset V(T)$ is $K$-\emph{nearly-regular} if either $d^-_T(v) \le d^+_T(v) \le Kd^-_T(v)$ for every $v \in B$, or $d^+_T(v) \le d^-_T(v) \le Kd^+_T(v)$ for every $v \in B$. The following lemma appears in~\cite{GPS}. We reproduce the proof here for convenience.

\begin{lemma}\label{lem:nearly-regular}
Any tournament $T$ contains a $4$-nearly-regular subset of vertices $R$ of size at least $\lfloor{|T|/20}\rfloor$. In particular, any $v \in R$ has in and out-degree at least $|T|/6$.
 \end{lemma}
 
\begin{proof}
Let $|T| = n$; we may assume $n \ge 20$, otherwise the result is trivial. Let $R \subset V(T)$ be the set of vertices for which either the ratio between the out-degree and in-degree or vice-versa is between $1$ and $4$. 
If $|R|\geq n/10$, then we are done, as we may pass to a subset $A \subset R$ of at least half the size of $R$ that is $4$-nearly-regular.
If $|R| < n/10$, then let $T'=T - R$, so that $|T'| >  9n/10$. Let $T'_1$ be the set of vertices $v \in V(T')$ for which $d_T^+(v) > 4d_T^-(v)$ and $T'_2$ be those vertices $v \in V(T')$ for which $d_T^-(v) > 4d_T^+(v)$. Suppose without loss of generality that $|T'_1|\geq |T'_2|$, so that $|T'_1|\geq \lceil{9n/20}\rceil$. Since each tournament on $t$ vertices has a vertex of in-degree $\ge \lceil (t-1)/2\rceil$, this implies that there is a vertex $u$ in $T'_1$ which has in-degree inside $T'_1$ at least $(\lceil{9n/20}\rceil - 1)/2$. However,
\[
    d_T^-(u) < \frac{1}{4}d_T^+(u) < \frac{1}{4}(n - d_T^-(u)),
\]
implying that $d_T^-(u) < n/5$. Thus, we obtain
\[
    (\lceil{9n/20}\rceil - 1)/2 \le d_T^-(u) < n/5,
\]
which yields a contradiction for $n \ge 20$.

Finally, let $R \subseteq V(T)$ be $4$-nearly-regular, and assume without loss of generality that $d^-_T(u) \le d^+_T(u) \le 4d^-_T(u)$ for all $u \in R$. If for some $v \in R$ we have $d^-_T(v) < n/6$, then  $d^+_T(v) < 2n/3$, and so $n - 1 = d^+_T(u) + d^-_T(u) < 5n/6$, a contradiction since $n \ge 20$.
\end{proof}

With these preliminaries in place, we shall prove our first lemma, guaranteeing strong $\ora{K}_r$ minors in $C'r\sqrt{\log r}$-strongly-connected tournaments. The strategy is to first embed a weak $\ora{K}_r$ minor in a `nice' portion of the tournament (using \Cref{thm:weak-minor}), and then make each branch set strongly-connected by linking appropriate pairs of vertices.

\begin{lemma}\label{lem:tournaments-connected}
Let $r$ be a positive integer. There exists a constant $C'>0$ such that if $T$ is a $C'r\sqrt{\log r}$-strongly-connected tournament, then $\strm(T) \ge r$.
\end{lemma}
\begin{proof}
  Put $C_0 = \max\{C, 24\}$ where $C$ is given by \Cref{thm:weak-minor}. Let $C' = 40C_0$, and suppose $T$ is a tournament that is $C'r\sqrt{\log r}$-strongly-connected. By \Cref{lem:nearly-regular} we can find a subset $S\subset V(T)$ of size at least $\lfloor{|T|/20}\rfloor \ge C_0r\sqrt{\log r}$ such that all vertices of $S$ have in and out degree at least $|T|/6$. Take a subset $S'\subset S$ of size exactly $C_0r\sqrt{\log r}$ (ignoring floors/ceilings) and let $T'=T - S'$. Find a weak $\ora{K}_r$ minor in $S'$ with branch sets $B_1, \ldots, B_r$, using \Cref{thm:weak-minor}.

 Since $T'$ is still $19C_0r\sqrt{\log r}$-strongly-connected, by \Cref{thm:linked}, $T'$ is $(19C_0/452)r\sqrt{\log r}$-linked. Hence, $T'$ is $r$-linked by our choice of $C_0$. We now construct a strong $\ora{K}_r$ minor as follows. For each $i \in [r]$, if $T[B_i]$ is already strongly-connected, then we take $B_i$ as a branch set. Otherwise, $B_i$ decomposes into strongly-connected components $X^i_1, \ldots, X^i_t$ for some $t \ge 2$, such that $X^i_p \rightarrow X^i_q$ for all $1 \le p < q \le t$. Let $s_i^1$ be any vertex in $X^i_1$ and let $s_i^2$ be any vertex of $X^i_t$. 
 
 Now, every vertex in $S'$ has many in and out-neighbours in $T'$. More precisely, every vertex in $S'$ has at least
 \[
    |T|/6 - |S'| \ge \frac{20C_0r\sqrt{\log r}}{6} - C_0r\sqrt{\log r} > C_0r\sqrt{\log r}
\]
in-neighbours and out-neighbours in $T'$. Thus, for each $s_i^1$ we may choose an in-neighbour $u_i^1 \in V(T')$, and for each $s_i^2$ we may choose an out-neighbour $u_i^2 \in V(T')$, such that the vertices $u_i^j$ are all pairwise distinct. In this way, we have disjoint sets $X = \{u_1^1, \ldots, u_s^1\}$ and $Y = \{u_1^2, \ldots, u_s^2\}$ where $s \le r$. Finally, use the fact that $T'$ is $r$-linked to find pairwise vertex-disjoint directed paths $P_i$ in $T'$ from $u_i^1$ to $u_i^2$ for each $i \in [s]$. Clearly, $T[B_i \cup V(P_i)]$ is strongly-connected for each $i \in [s]$, and thus we have found a strong $\ora{K}_r$ minor in $T$.
\end{proof}

\Cref{lem:tournaments-connected} is tight up to the constant $C$, by considering a random tournament $T$ on $cr\sqrt{\log r}$ vertices. As mentioned in the Introduction, following the argument of Bollob{\'a}s, Catlin, and Erd\H{o}s \cite{BCH}, w.h.p. it contains no strong $\ora{K}_r$ minor (in fact, it contains no $r$ pairwise disjoint nonempty sets with edges in each direction between each pair of sets). Additionally, it is not hard to show that w.h.p.  $T$ is $(cr\sqrt{\log r})/10$-strongly-connected. This can be seen by noting that w.h.p. every two vertices have at least approximately $|T|/10$ directed paths of length $2$ in each direction. 

The following two lemmas allow us to find some structure in a tournament with large minimum out-degree, under the additional assumption that it is not highly strongly-connected. We use the following notation in their formulations. If $\cH$ is a family of subdigraphs of $D$, let $\bigcup \cH$ denote the set $\bigcup_{H \in \cH} V(H)$.

\begin{lemma}\label{lem:triv}
Let $m \ge 2$ and $d \ge 1$ be integers and suppose $T$ is a vertex-minimal tournament with the property that $\delta^{+}(T)\geq d$. Then there exists a positive integer $t \le m - 1$ such that the following holds. If $T$ is not $m$-strongly-connected, then we can find a non-empty set $S=\{w_1v_1,\ldots,w_tv_t\}$ of pairwise vertex-disjoint directed edges and a subtournament $T' \subseteq T$ with the following properties:
\begin{itemize}
    \item $V(T')\cap (\bigcup S) = \varnothing$.
    \item For any $i \in [t]$ we have $|N^+(v_i)\cap T'| \ge |T'| - m + 1$, and $|N^-(w_i)\cap T'| \ge |T'| - d + 1$.
    \item $\delta^+(T') \ge d - 2t$.
\end{itemize}
\end{lemma}
\begin{proof}
Assume that $T$ is not $m$-strongly-connected and let $R$ be a smallest cut-set of size at most $m-1$. Observe that $R \neq \varnothing$: if $T$ itself is not strongly-connected, then consider the source set $X$ and sink set $Y$ of $T$. Since $X \rightarrow Y$, we must have $\delta^+(T[Y]) \ge d$, contradicting the minimality of $T$.

Let $A$ and $B$ denote the source set and sink set of $T - R$, respectively, and define
\[
    R' =\{x\in R: |N^+(x) \cap (B \cup R)| <d\}.
\]
The minimality of $T$ implies that $R'\neq \varnothing$: if $R' = \varnothing$, then $\delta^+(T[B\cup R]) \ge d$, a contradiction. Now, choose a maximum matching $M = \{w_1v_1, \ldots, w_tv_t\}$ from $R'$ to $A$ and let $T' = T[(B\cup R)\setminus V(M)]$ and $S = M$. We claim that $T'$ and $S$ satisfy the required properties. Clearly, the first property is satisfied. To see the last property suppose that $x\in B\cup (R\setminus R')$, then its degree is at least $d-t$ in $T'$. On the other hand, if $x\in R'\setminus V(M)$, since we can not extend the matching $M$, $N_{T}^{+}(x)\cap A\subset V(M)$, which implies $d_{T'}^+(x)\geq d-|V(M)|=d-2t$. To see the second property, consider a vertex $v_i$. Then $v_i \rightarrow B$, and so there are at most $|R|\leq m - 1$ vertices in $|T'|$ that are in-neighbours of $v_i$. Similarly, consider a vertex $w_i$. Since $w_i \in R'$, by definition $|N^+(x) \cap (B \cup R)| < d$. Thus, $w_i$ has at least $|T'| - d + 1$ in-neighbours in $T'$.
\end{proof}

Our final lemma is obtained from iteratively applying \Cref{lem:triv} under the assumption that no subtournament is highly strongly-connected.

\begin{lemma}\label{claim:main}
Suppose $m \ge 2$ and $d \ge 1$ are integers, and suppose $T$ is a tournament such that $\delta^+(T) \ge d$ and no subtournament of $T$ is $m$-strongly-connected. Then there exists a set $S = \{w_iv_i: i \in [m]\}$ composed of pairwise vertex-disjoint directed edges and a set $F(S) \subseteq V(T)$ such that for every $i \in [m]$ 
\begin{enumerate}
    \item $F(S)\cap (\bigcup S) =\varnothing$ and $|F(S)|\geq 2(d - 4m)$,\label{itm:main-1}
    \item $|N^+(v_i)\cap F(S)| \ge |F(S)| - m$, and \label{itm:main-2}
    \item $|N^-(w_i)\cap F(S)| \ge |F(S)| - d$.\label{itm:main-3}
\end{enumerate}
\end{lemma}
\begin{proof}
 	 We iteratively apply \Cref{lem:triv}. Begin by removing vertices from $T$ such that the remaining tournament $T_0$ is vertex-minimal with $\delta^+(T_0) \ge d$. Then, since $T_0$ is not $m$-strongly-connected by assumption, \Cref{lem:triv} implies that there is a positive integer $t_1 < m$, a set of edges $S_1 = \{w^1_iv^1_i : i \in [t_1] \}$, and a subtournament $F_1 \subseteq T_0$ such that 
 	\begin{itemize}
 		\item $V(F_1) \cap (\bigcup S_1) = \varnothing$,
 		\item $\delta^+(F_1) \ge  d - 2t_1$, and
 		\item $|N^+(v_i^1) \cap F_1| \ge |F_1| - m $ and $|N^-(w_i^1) \cap F_1| \ge |F_1| - d$ for all $i \in [t_1]$.
 	\end{itemize} 
 We proceed by applying  \Cref{lem:triv} inside $F_1$. Now assume we have already found subtournaments $F_1 \supset \cdots \supset F_k$ with $\delta^+(F_k) \ge d -2 \sum_{i=1}^k t_i$, and disjoint sets of edges $S_1, \ldots, S_k$ of sizes $t_1, \ldots, t_k$, such that for each $j \in [k]$ we have
 	$|N^+(v_i^j) \cap F_k| \ge |F_k| - m $ and $|N^-(w_i^j) \cap F_k| \ge |F_k| - d$ for all $i \in [t_k]$. 
 	
 	Assume that $\sum_{i=1}^k t_i < m$.
 	We first find a vertex-minimal subtournament $T_k$ of $F_k$ such that $\delta^+(T_k) \ge d -2\sum_{i=1}^k t_i$. Since by assumption $T_k$ is not $m$-strongly-connected, 
 	 \Cref{lem:triv} implies that  there is $t_{k+1} < m$, a non-empty set of edges $S_{k+1} = \{w_i^{k+1}v_i^{k+1} : i \in [t_{k+1}]\}$, and a subtournament $F_{k+1} \subseteq T_{k}$ such that 
 	\begin{itemize}
 		\item $V(F_{k+1}) \cap (\bigcup S_{i}) = \varnothing$ for $i \in [k+1]$,  
 		\item $\delta^+(F_{k+1}) \ge  \delta^+(T_k) - 2t_{k+1} \ge  d - 2\sum_{i=1}^{k+1} t_i$,
 		\item $|N^+(v_i^j) \cap F_{k+1}| \ge |F_{k+1}| - m$ and $|N^-(w_i^j) \cap F_{k+1}| \ge |F_{k+1}| - d$ for $j \in [k+1]$ and $i \in [t_j]$.
 	\end{itemize} 
 	We stop this process at a step $k_0$ when $\sum_{i=1}^{k_0} t_i \ge m$, and let 
 	\[
 	    S = \bigcup_{i=1}^{k_0} S_i \quad \text{ and } \quad F(S) = F_{k_0}.
 	   \]
 Then clearly properties (\ref{itm:main-2}) and (\ref{itm:main-3}) of the claim are satisfied, and we have a set $S$ of at least $m$ edges which are all disjoint from $F(S)$. We just need to show that $|F(S)| \ge 2(d - 4m)$. Since $\sum_{i=1}^{k_0-1}t_i < m$ and $t_{k_0} < m$, we have that $\sum_{i=1}^{k_0} t_i < 2m$, and so $\delta^+(F_{k_0}) \ge d - 4m$. Thus,
\[
|F(S)| = |F_{k_0}| \ge 2\delta^+(F_{k_0}) \ge 2(d - 4m),
\]
as claimed.
\end{proof}

We are now in position to prove \Cref{thm:main-tourn-2}. 

\begin{proof}[Proof of \Cref{thm:main-tourn-2}]
As we already mentioned, the tightness can be seen by considering a random tournament on $\Omega(r\sqrt{\log r})$ vertices. 

Now, let $C'$ be the constant given by \Cref{lem:tournaments-connected} and let $c$ be the constant given by \Cref{thm:kostochka-thomason}. Put $C_0' = \max\{C', 62c\}$ and $C = 1600C_0'$. Define $m = C'r\sqrt{\log r}$ and $d = Cr\sqrt{\log r}$, and suppose $T$ is a tournament with $\delta^+(T) \ge d$. Note that $C\geq 1600C'$.  If $T$, or any subtournament of $T$, is $m$-strongly-connected, then we are done by \Cref{lem:tournaments-connected}; so we may assume otherwise. In this situation, we may apply \Cref{claim:main}. Indeed, let $S$ and $F(S)$ be given as in \Cref{claim:main}, and let $N_i=N^{-}(w_i)\cap F(S)$, $M_i=N^{+}(v_i)\cap F(S)$, and $L_i=M_i\cap N_i$. We first observe that for each $i \in [m]$
\begin{equation}\label{eqn:Li}
   | L_i| \ge |F(S)| - d - m \ge d/2 - m \ge d/4,
\end{equation}
using the fact that $|F(S)| \ge 2(d - 4m)$.
For every $i \in [m]$ choose a vertex $z_i$ from $L_i$ uniformly at random with replacement and let $L$ denote the resulting random set of vertices. Note that $\Delta_i :=\{w_i,v_i,z_i\}$ forms a directed triangle. Let $X$ denote the random variable recording the size of a largest subset of $L$ in which all vertices are pairwise distinct. Further, let $X'$ count the number of pairs $\{z_i, z_j\}$ with $i \neq j$ and $z_i = z_j$. We have that $\mathbb{P}(z_i = z_j) = |L_i \cap L_j|/ |L_i||L_j| \le 4/d$ using (\ref{eqn:Li}). It follows that $\mathbb{E}[X'] \le \binom{m}{2}\frac{4}{d} \le 2m^2/d$, and hence
\begin{equation}\label{eqn:exp-X}
    \mathbb{E}[X] \ge m - \mathbb{E}[X'] \ge m - 2m^2/d = (C'- 2C'^2/C) r\sqrt {\log r} \geq  (C'/2) r\sqrt{\log r} = m/2,
\end{equation}
since $C\geq 1600C'$. 

We say that a pair $\{\Delta_i, \Delta_j\}$ of directed triangles is \emph{good} if there are edges between them in both directions; otherwise, we say this pair is \emph{bad}. Observe that
\begin{align*}
\Pro(\{\Delta_i, \Delta_j\} \text{ is bad }) &\le \Pro(z_j \notin N^{+}(v_i)\cap F(S) \text{ or } z_i \notin N^+(v_j)\cap F(S)) \\
&\le \frac{m}{|L_j|} + \frac{m}{|L_i|} \le \frac{8m}{d} = \frac{8C'}{C} \leq  \frac{1}{200},
\end{align*}
using (\ref{eqn:Li}) and since $C \geq 1600C'$. Hence, letting $Y$ denote the random variable counting the number of bad pairs of directed triangles, we have that
\begin{equation}\label{eqn:exp-Y}
    \mathbb{E}[Y] \le \binom{m}{2}\frac{1}{200} \le m^2/400.
\end{equation}
 
Combining (\ref{eqn:exp-X}) and (\ref{eqn:exp-Y}), and using the Cauchy-Schwarz inequality, we have
\[
\mathbb{E}\left[X^2 - 40Y - \frac{m^2}{9}\right] \ge \mathbb{E}[X]^2 - 40\mathbb{E}[Y] - \frac{m^2}{9}
\ge \frac{m^2}{4} - \frac{m^2}{10} - \frac{m^2}{9} > 0.
\]

Accordingly, there is a choice of vertices from $\bigcup_{i =1}^m L_i$ such that $X^2 - 40Y - m^2/9 > 0$. It follows that $X^2 \ge m^2/9$, and so $X \ge m/3$; i.e., there is a set $U$ of at least $m/3$ vertices and they are all pairwise distinct. By possibly passing to a subset, let us assume $|U| = \lfloor{m/3}\rfloor$. Moreover, we must have $40Y \le X^2 \le m^2$, and so the number of bad pairs of triangles is at most $m^2/40$.

Our final aim is define a suitable auxiliary graph on our directed triangles and apply \Cref{thm:kostochka-thomason}. Define a graph $G$ in the following way: put $V(G) = \{\Delta_i: z_i \in U\}$ and join $\Delta_i$ to $\Delta_j$ if and only if the pair $\{\Delta_i, \Delta_j\}$ is good. By our above analysis, we have $|V(G)| = \lfloor{m/3}\rfloor$, and there are at most $m^2/40$ non-edges in $G$. Thus, there are at least
\[
    \binom{\lfloor{m/3}\rfloor}{2} - m^2/40 \ge m^2/36 - m^2/40 \ge .0027 m^2
\]
edges in $G$, and so the average degree is at least $.0162 m \ge .0162(62c r\sqrt{\log r}) > cr\sqrt{\log r}$. Here we are using the fact that $m = C'r\sqrt{\log r}$ and $C' \ge 62c$ by definition. Applying \Cref{thm:kostochka-thomason} to $G$, we have that $G$ contains a $K_r$ minor. It is easy to see that this $K_r$ minor corresponds to a strong $\ora{K}_r$ minor in the original tournament. This completes the proof of \Cref{thm:main-tourn-2}.
\end{proof}

%%%%%%%%%%%%%%%%%%%%%%%%%%%%%%%%%%%%%%%%%%%%%%%%%%%%%%%%%%%%%%%%%%%%%%%
%%%%%%%%%%%%%%%%%%%%%%%%%%%%%%%%%%%%%%%%%%%%%%%%%%%%%%%%%%%%%%%%%%%%%%%

\setcounter{claim}{0}

\section{Strong minors in digraphs}\label{sec:digraphs}

We already know that large out-degree is, in general, not sufficient to guarantee large complete minors in digraphs. In this section, we give a positive result in digraphs for dichromatic number. More precisely, we prove \Cref{thm:main-digraph}, showing that there is a function $f(r)$ such that any digraph with dichromatic number at least $f(r)$ contains a strong $\ora{K}_r$ minor. First, we need a few definitions. A \emph{$\ora{K}_r$-template} is a digraph $D$ such that $V(D)$ admits a partition into sets $X_1, \ldots, X_r$ such that for every ordered pair $(i, j)$ with $i \neq j \in [r]$, there is an edge from $X_i$ to $X_j$. We say that a digraph contains a $\ora{K}_r$-template if it contains a $\ora{K}_r$-template as a subdigraph. It is easy to find $\ora{K}_r$-templates in digraphs with high dichromatic number:
\begin{lemma}\label{lem:easy-transitive}
Let $r \ge 1$ be an integer and suppose $D$ is a digraph with $\chi(D) \ge r$. Then $D$ contains a $\ora{K}_r$-template.
\end{lemma}
\begin{proof}
Let $A$ be a maximum sized transitive set in $D$ and let $D' = D - A$. Then every vertex in $D'$ has an out-neighbour and an in-neighbour in $A$ (otherwise, we could create a larger transitive set in $D$). Moreover, $\chi(D') \ge r - 1$. The result follows by induction on $r$, by taking the union of a $\ora K_{r-1}$-template in $D'$ together with the set $A$.
\end{proof}

We are now ready to prove \Cref{thm:main-digraph}. The proof follows the ideas of Aboulker et al.~\cite{ACH} for embedding subdivisions of $\ora{K}_r$ in digraphs of high dichromatic number.

For a strongly-connected digraph $D$ we define an \emph{out-BFS-tree} $T^+ = T^+_v$ in $D$ with root $v$  as a subdigraph of $D$ spanning $V(D)$, such that $T^+_v$ is an oriented tree and for every $w \in V(D)$ we have $\dist_{T^+}(v,w) = \dist_D(v,w)$. 
Similarly,   an  \emph{in-BFS-tree} $T^- = T^-_v$ rooted at $v$ is a subdigraph of $D$ spanning $V(D)$ which is an oriented tree, such that for every $w \in V(D)$ we have 
$\dist_{T^-}(w,v) = \dist_D(w,v)$. For every integer $i$ and for a vertex $v \in V(D)$, let $L_i^{+,v}$ denote the $i$th out-layer from $v$ in $D$, i.e. $L_i^{+,v} = \{w \in V(D) : d_D(v,w) = i \}$. Similarly, $L_i^{-,v}$ denotes the $i$th in-layer from $v$ in $D$, i.e. $L_i^{-,v} = \{w \in V(D): d_D(w,v) = i \}$. 

We need the following simple fact:

%\begin{claim}\label{claim:large-chi}
%	Let $D$ be a strongly-connected digraph and let $T$ be an in- or out-BFS-tree in $D$. Then there is a level $L$ such that $\chi(D[L]) \ge \chi(D)/2$.
%\end{claim}
%\begin{proof}
%	Without loss of generality, suppose that $T$ is an out-BFS-tree in $D$ with levels $L_1,\ldots,L_m$. Let $D_1$ and $D_2$ be the subdigraphs of $D$ induced by the odd and even levels, respectively. Since there is no arc from $L_i$ to $L_j$ for $j > i+1$, the strongly-connected components of $D_1$ and $D_2$ must be contained within the levels. It follows that $\chi(D_1) = \max\limits_{i \text{ odd}} \chi(D[L_i]) $ and $\chi(D_2) = \max\limits_{i \text{ even}} \chi(D[L_i]) $. Since $D$ is strongly-connected, we have $V(D) =V(D_1) \cup V(D_2)$ and thus, $\chi(D) \le \chi(D_1) + \chi(D_2) \le 2 \max\limits_{i \in [m]}\chi(L_i)$.
%\end{proof}

\begin{lemma}\label{claim:large-chi}
	Let $D$ be a strongly connected digraph and let $T$ be an in- or out-BFS-tree in $D$. Let $X$ be any subset of vertices of $D$. Then there is a layer $L$ of  $T$  such that $\chi(D[L\cap X]) \ge \chi(D[X])/2$.
\end{lemma}

\begin{proof}
 Without loss of generality, suppose that $T$ is an out-BFS-tree in $D$ with layers $L_1,\ldots,L_m$. Let $D_1$ and $D_2$ be the subdigraphs of $D$ induced by the odd and even layers of $T$, respectively. Since there is no arc from $L_i$ to $L_j$ for $j > i+1$, the strongly-connected components of $D_1$ and $D_2$ must be contained within the layers. It follows that $\chi(D_1[X]) = \max\limits_{i \text{ odd}} \chi(D[L_i\cap X]) $ and $\chi(D_2[X]) = \max\limits_{i \text{ even}} \chi(D[L_i\cap X]) $. 
	Since $D$ is strongly-connected, we have $V(D) =V(D_1) \cup V(D_2)$ and and thus, $\chi(D[X]) \le \chi(D_1[X]) + \chi(D_2[X]) \le 2 \max\limits_{i \in [m]}\chi(L_i\cap X)$.
\end{proof}

\begin{proof}[Proof of \Cref{thm:main-digraph}]
	Let $D$ be a digraph such that $\chi(D) \geq r4^r$.  We shall show that ${\strm}(D) \geq r$. We may assume that $D$ is strongly-connected, otherwise, by \Cref{lem:chi-subgraph}, pass to a strongly-connected component of high dichromatic number. For an integer $m$ with $0 \le m \le r$, an $m$-{\it partial strong } $\ora{K}_r$ minor in a digraph $D$ is a collection of pairwise disjoint subsets $X_1, \ldots, X_r \subseteq V(D)$ such that
\begin{itemize}
    \item for every pair $(i, j)$ with $i \neq j \in [r]$, there exists an edge from $X_i$ to $X_j$, and
    \item $D[V_i]$ is strongly-connected for at least $m$ sets $V_i$. 
\end{itemize} 
Note that a $\ora{K}_r$-template is a $0$-partial strong $\ora{K}_r$ minor. We define $f_r(m)$ to be the smallest $k$ such that any digraph with $\chi(D) \ge k$ contains an $m$-partial strong $\ora{K}_r$ minor. Note that by \Cref{lem:easy-transitive}, $f_r(0) \le r$.
\begin{claim}\label{claim:ind} For any $m$,  $0 \le m \le r$,    $f_r(m) \le r 4^m$.
\end{claim}

\Cref{claim:ind} implies in particular that any digraph with dichromatic number at least $r4^r$ has a strong $\ora{K}_r$ minor. 
We shall prove \Cref{claim:ind} by induction on $m$. 
Let $\chi(D) \ge r 4^m$, $0 \le m \le r$. If $m = 0$, then the claim follows by \Cref{lem:easy-transitive}. So assume that $1 \le m \le r$ and that the claim holds for smaller $m$. 
Let $v \in D$, let $T^-$, $T^+$ be the in- and out-BFS-trees rooted at $v$, respectively.
% and let  $L^-_{1},\ldots, L^-_p$ denote the in-layers, and $L^+_1,\ldots, L^+_q$ denote the out-layers of $T^-, T^+$. 

Apply Lemma \ref{claim:large-chi} to $D$ with $T=T^-$ and $X=V(D)$ first to find a layer $L$ in $T=T^-$ such that $\chi(D[L])\geq \chi(D)/2$. 
Then apply Lemma \ref{claim:large-chi} to $D$  with $T=T^+$ and $X=L$ to find a layer $L'$ in $T^+$ such that 
$\chi(D[L'\cap L])\geq \chi(D[L])/2$.   Thus we have that $\chi(D[L'\cap L]) \geq \chi(D)/4 \geq r 4^{m-1}$.
By induction, we have that $f_r(m-1) \leq r4^{m-1}$, thus 
%\begin{claim}\label{claim:in-out-large-chi}
%There exist $i \in [p], j \in [q]$ with $\chi(D[L_i^- \cap L_j^+] ) \ge \chi(D)/4 \ge f_r(m-1)$.
%\end{claim}
%\begin{proof}
%By \Cref{claim:large-chi}, there is $i \in [p]$, such that $\chi(D[L^-_i]) \ge \chi(D)/2$. Choose a subset $V' \subset L_i^-$ such that $D' = D[V']$ is strongly-connected with $\chi(D') = \chi(D[L_i^-]) \ge \chi(D)/2$, which exists by \Cref{lem:chi-subgraph}.
%Let $T'^+ = T^+[V']$, i.e. the out-BFS-tree rooted at $v$ restricted to the subset $D'$, with layers $L'^+_k = L_k^+ \cap V'$, $k \in [q]$. Since $T^+$ is an out-BFS-tree,  the strong components are contained within the layers $L'^+_k$. Thus, there is $j \in [q]$, such that $\chi(D[L'^+_j]) \ge \chi(D')/2 \ge \chi(D)/4$ with $L'^+_j \subseteq L_i^- \cap L_j^+$.
%\end{proof}
we can find an $(m-1)$-partial strong $\ora K_r$ minor $K$ contained in $L'\cap L$ with branch sets $V_1,\ldots, V_r$.
If all sets $V_i$ induce a strongly connected digraph, we have that $K$ is a strong $\ora K_r$-minor and $\strm(D)\geq r$, so we are done.
Otherwise assume without loss of generality that  $D[V_1]$ is not strongly-connected. Let $V_1 = \{x_1,\ldots, x_s\}$. %, where $s \le 2(r-1)$. 
Now for each $\ell \in [s]$ let $P_\ell^+$ be the directed $v-x_\ell$ path in $T^+$ and $P_\ell^-$ the directed $x_\ell-v$ path in $T^-$.
Note that these paths ``leave" the respective layers in $T^+$ and $T^-$ immediately.  Then for $\ell \in [s]$, $\{x_\ell\} \cup (P_\ell^-) \cup (P_{\ell+1}^+) \cup \{x_{\ell+1}\}$ (with addition modulo $s$) induces an $x_\ell-x_{\ell+1}$ walk, which contains an $x_\ell - x_{\ell+1}$ path. This path intersects $K$ exactly in the two points $ \{x_\ell, x_{\ell+1}\}$, since $V(K)$ is contained in a layer $L$ of $T^-$ and in a subset $L'$ of a layer of $T^+$. 
Thus, letting $V_1' = V_1 \cup \bigcup_{\ell=1}^s P_\ell^+ \cup P_\ell^-$, we see that $D[V_k']$ is strongly-connected and intersects $K$ only in $V_1$. Hence
we obtain an $m$-partial strong $\ora K_r$ minor $K'$ with branch sets $V'_1, V_2, \ldots,  V_r$. 
This proves the claim and the theorem.
\end{proof}

%%%%%%%%%%%%%%%%%%%%%%%%%%%%%%%%%%%%%%%%%%%%%%%%%%%%%%%%%%%%%
%%%%%%%%%%%%%%%%%%%%%%%%%%%%%%%%%%%%%%%%%%%%%%%%%%%%%%%%%%%%%

\section{Concluding remarks and open problems}\label{sec:final}

We have investigated several  relationships  between large dichromatic number and the presence of strong complete minors in digraphs. Many problems remain. Regarding \Cref{thm:main-tourn-1}, the most obvious question is whether or not the bound on dichromatic number can be decreased.

\begin{question}
Is it true that any tournament $T$ with $\chi(T) \ge (1 + o(1)) r$ satisfies $\strm(T) \ge r$?
\end{question}

For general digraphs, our bounds are far apart. To take a more general view, we may address the following question: Which digraph parameters force the existence of large strong complete minors? To formalize this, for a digraph $H$ and a digraph parameter $\phi$, we let $\strm_\phi(H)$ denote the smallest integer such that any digraph $D$ with $\phi(D) \ge \strm_\phi(H)$ contains a strong $H$ minor; if such an integer does not exist, then we define $\strm_\phi(H) = \infty$. This is analogous to the work of Aboulker et al.~\cite{ACH}, where they introduced the parameter $\text{mader}_\phi(H)$ as the smallest integer such that any digraph $D$ with $\phi(D) \ge \text{mader}_\phi(H)$ contains a subdivision of $H$. \Cref{thm:negative-digraph} can be expressed as saying that for any $r$, $\strm_{\delta^+}(\ora{K}_r) = \infty$, and  our \Cref{thm:main-digraph} can be expressed succintly as
\[
    r + 1 \le \strm_{\chi}(\ora{K}_r) \le r4^r.
\]
It would be very interesting to obtain better bounds on this function. 

In \Cref{lem:tournaments-connected}, we showed that strong-connectivity is sufficient to guarantee large complete minors in tournaments. Is the same true in general digraphs? Letting $\kappa(D)$ denote the strong-connectivity of a digraph $D$, we pose the following problem:

\begin{problem}
Determine whether or not $\strm_{\kappa}(\ora{K}_r) < \infty$.
\end{problem}

For a given natural digraph parameter $\phi$, it would be interesting to determine  $\strm_\phi(H)$, for some other digraphs $H$,  for example oriented trees, cycles, or transitive tournaments.

%\bibliographystyle{amsplain}
%\bibliography{minors.bib}

\end{document}